\documentclass{article}
\usepackage{amssymb}
\usepackage{amsmath}
\usepackage{amsfonts}

\newtheorem{theorem}{Theorem}[section]

\newtheorem{lemma}{Lemma}[section]

\newtheorem{remark}[theorem]{Remark}

\numberwithin{equation}{section}
\newenvironment{proof}[1][Proof]{\noindent\textbf{#1.} }{\ \rule{0.5em}{0.5em}}

\begin{document}

\title{General enegy decay for a viscoelastic equation of Kirchhoff type
with acoustic boundary conditions}
\author{Abdelkader MAATOUG \\
Loboratory of Energetic Engineering and Computer Engineering\\
Tiaret University\\
P.O.Box 78, Zaaroura, Tiaret 14000, Algeria.\\
E-mail: abdelkader.maatoug@univ-tiaret.dz\\
}
\date{}
\maketitle

\begin{abstract}
This paper is concerned with a viscoelastic equation of Kirchhoff type with
acoustic boundary conditions in a bounded domain of $\mathbb{R}^{n}.$ We
show that, under suitable conditions on the initial data, the solution
exists globally in time. Then, we prove the general energy decay of global
solutions by applying a lemma of P. Martinez, wihch allows us to get our
decay result for a class of relaxation functions wider than that usually
used.\newline

\end{abstract}

\section{Introduction}

In this paper we are concerned with the decay rates of solutions for the
following nonlinear wave equation of Kirchhoff type, with acoustic boundary
conditions

\begin{equation}
\left\{ 
\begin{tabular}{ll}
$u_{tt}-\left( a+b\left\Vert \nabla u\right\Vert _{2}^{2\gamma }\right)
\Delta u+\int_{0}^{t}g\left( t-s\right) \Delta u\left( s\right)
ds=|u|^{k-2}u,$ & in $\Omega \times (0,+\infty ),$ \\ 
$u=0,$ & on $\Gamma _{0}\times (0,+\infty ),$ \\ 
$\left( a+b\left\Vert \nabla u\right\Vert _{2}^{2\gamma }\right) \frac{%
\partial u}{\partial \nu }-\int_{0}^{t}g\left( t-s\right) \frac{\partial u}{%
\partial \nu }\left( s\right) ds=y_{t}$ & on $\Gamma _{1}\times (0,+\infty
), $ \\ 
$u_{t}+p\left( x\right) y_{t}+q\left( x\right) y=0$ & on $\Gamma _{1}\times
(0,+\infty ),$ \\ 
$u(.,0)=u_{0},u_{t}(.,0)=u_{1},$ & in $\Omega $, \\ 
$y(.,0)=y_{0},$ & on $\Gamma _{1}$,%
\end{tabular}%
\right.  \label{Prob}
\end{equation}%
where $\Omega $ is a bounded domain of $\mathbb{R}^{n},$ $n\geq 1$ having a
smooth boundary $\Gamma =\partial \Omega $, consisting of two closed and
disjoint parts $\Gamma _{0}$ and $\Gamma _{1}$. Here $\nu $ denotes the unit
outward normal to $\Gamma $. The parameters $a>0,\ b>0,$ $\gamma \geq 0$ and 
$k>2$ are constant real numbers, $p\ $and $q$ are given functions satisfying
some conditions to be specified later; $u_{0},u_{1}:\Omega \rightarrow 
\mathbb{R}$ and $\ y_{0}:\Gamma _{1}\rightarrow \mathbb{R}$ are given
functions.\medskip

The acoustic boundary conditions were introduced by Beale and Rosencrans in 
\cite{Beale1}, \cite{Beale2}; where the authors proved the global existence
and regularity of solutions of the wave equation subject to boundary
conditions of the form 
\begin{equation}
\left\{ 
\begin{tabular}{ll}
$\frac{\partial u}{\partial \nu }=y_{t}$ & on $\Gamma \times (0,+\infty ),$
\\ 
$\rho u_{t}+m\left( x\right) y_{tt}+p\left( x\right) y_{t}+q\left( x\right)
y=0$ & on $\Gamma \times (0,+\infty ),$%
\end{tabular}%
\right.  \label{beale-cond}
\end{equation}%
where $\rho >0$, $p$ is a nonnegative function and $m$, $q$ are strictly
positive functions on the boundary.

Wave equations with acoustic boundary conditions have been treated by many
authors \cite{Frota2005}, \cite{Mugnolo06}, \cite{Vicente09}, \cite{Park11}, 
\cite{Boukhatem}, \cite{Park acoustic 16} . In \cite{Frota2005}, the authors
studied a linear wave equation subject to the boundary conditions (\ref%
{beale-cond}), with $m\equiv 0$, on the portion $\Gamma _{1}$ of the
boundary and Dirichlet boundary conditions on the portion $\Gamma _{0}.$
They proved global solvability and uniform energy decay.

Boukhatem and Benabderrahmane \cite{Boukhatem}, studied the
variable-coefficient viscoelastic wave equation 
\begin{equation}
\left\{ 
\begin{array}{ll}
u_{tt}+Lu-\int_{0}^{t}g(t-s)Lu(s)ds=|u|^{k-2}u, & \text{in }\Omega \times
(0,\infty ), \\ 
u=0, & \text{on\ }\Gamma _{0}\times \ (0,\infty ), \\ 
\frac{\partial u}{\partial \nu _{L}}-\int_{0}^{t}g(t-s)\frac{\partial u}{%
\partial \nu _{L}}(s)ds=h(x)y_{t}, & \text{on }\Gamma _{1}\times \ (0,\infty
), \\ 
u_{t}+p\left( x\right) y_{t}+q\left( x\right) y=0, & \text{on }\Gamma
_{1}\times (0,\infty ),%
\end{array}%
\right.  \label{Boukhat-14}
\end{equation}%
where $Lu=-div(A\nabla u)$, $\frac{\partial u}{\partial \nu _{L}}=\left(
A\nabla u\right).\nu $ and $A=\left( a_{ij}(x)\right) $ a matrix with $%
a_{ij}\in C^{1}(\overline{\Omega })$. Combining the techniques used by
Georgiev and Todorova in \cite{Todorova}, those used by Frota and Larkin in 
\cite{Frota2005}, and Faedo--Galerkin's approximations, the authors proved
global solvability for suitable initial data, and general energy decay for
some relaxation functions $g$ satifaying some known condtions, introduced
firstly by Messaoudi in \cite{Messaoudi08}.\medskip

For quasilinear equations, the authors of \cite{Frota2000} studied the
Carrier equation 
\begin{equation}
\begin{tabular}{ll}
$u_{tt}-M\left( \left\Vert u\right\Vert _{2}^{2}\right) \Delta u+\left\vert
u_{t}\right\vert ^{\alpha }u_{t}=f(u),$ & in $\Omega \times (0,+\infty ),$%
\end{tabular}
\label{Pb-Frota-00}
\end{equation}%
subject to the boundary conditions (\ref{beale-cond}) on the portion $\Gamma
_{1}$ and Dirichlet boundary conditions on the portion $\Gamma _{0},$ where $%
M:\mathbb{R}_{+}\longrightarrow \mathbb{R}$ is a $C^{1}$ function satisfying
some conditions. They proved the existence and uniqueness of global
solutions. In \cite{Yong2012}, the author studied the following problem 
\begin{equation}
\left\{ 
\begin{tabular}{ll}
$u_{tt}-M\left( \left\Vert \nabla u\right\Vert _{2}^{2}\right) \Delta
u+2\delta u_{t}=0$ & in $\Omega \times (0,+\infty ),$ \\ 
$u=0,$ & on $\Gamma _{0}\times (0,+\infty ),$ \\ 
$M\left( \left\Vert \nabla u\right\Vert _{2}^{2}\right) \frac{\partial u}{%
\partial \nu }=y_{t}$ & on $\Gamma _{1}\times (0,+\infty ),$ \\ 
$u_{t}+p\left( x\right) y_{t}+q\left( x\right) y=0$ & on $\Gamma _{1}\times
(0,+\infty ),$%
\end{tabular}%
\right.  \label{Pb-Yong-12}
\end{equation}%
where $M\left( \left\Vert \nabla u\right\Vert _{2}^{2}\right) =a+b\left\Vert
\nabla u\right\Vert _{2}^{2}$, with $a,b>0.$ He proved the uniform stability
of solutions for a sufficiently small positif passive viscous damping
coeficient $\delta $, using multiplier technique.\newline
Recently, Lee et all in \cite{Park acoustic 16} were concerned with the
following Kirchhoff type equation with Balakrishnan-Taylor damping,
time-varying delay and boundary conditions:%
\begin{equation}
\left\{ 
\begin{tabular}{ll}
$\left\vert u_{t}\right\vert ^{\rho }u_{tt}-\left( M\left( u\right) \left(
t\right) \right) \Delta u-\Delta u_{tt}$ &  \\ 
\ \ \ \ \ $+\int_{0}^{t}g\left( t-s\right) \Delta u\left( s\right)
ds+|u_{t}|^{q}u_{t}=|u|^{p}u$ & in $\Omega \times (0,+\infty ),$ \\ 
$u=0,$ & on $\Gamma _{0}\times (0,+\infty ),$ \\ 
$\left( M\left( u\right) \left( t\right) \right) \frac{\partial u}{\partial
\nu }+\frac{\partial u_{tt}}{\partial \nu }-\int_{0}^{t}g\left( t-s\right) 
\frac{\partial u}{\partial \nu }\left( s\right) ds$ &  \\ 
\ \ \ \ \ \ \ \ $+\mu _{0}u_{t}\left( x,t\right) +\mu _{1}u_{t}\left(
x,t-\tau \left( t\right) \right) =h\left( x\right) y_{t}$ & on $\Gamma
_{1}\times (0,+\infty ),$ \\ 
$u_{t}+p\left( x\right) y_{t}+q\left( x\right) y=0$ & on $\Gamma _{1}\times
(0,+\infty ),$ \\ 
$u_{t}\left( x,t-\tau \left( t\right) \right) =f_{0}(x,t)$ & on $\Gamma
_{1},-\tau \left( 0\right) \leq t\leq 0$,%
\end{tabular}%
\right.  \label{Pb-Park 16}
\end{equation}%
where $M\left( u\right) \left( t\right) =a+b\left\Vert \nabla u\right\Vert
_{2}^{2}+\sigma \left( \nabla u,\nabla u_{t}\right) $, with $a,b,\sigma >0.$
They showed the global existence of solutions and established a general
energy decay \bigskip

Motivated by the previous works, we consider a non degenerate Kirchhoff type
wave equation with memory term, acoustic boundary conditions and source
term. This model is new because has not been considered by predecessors.
Further, we do not use differetial inequalities to prove our decay result,
but we use a method based on integral inequalities, introduced by P.
Martinez \cite{Martinez} (Lemma \ref{lemme Martinez} below) and that
generalize those introduced by V. Komornik \cite{Komornik 96} and A. Haraux 
\cite{Haraux 85}. This method allows us to consider a class of relaxation
functions larger than the one usually considered.

The paper is organized as follows. In Section 2, we present some notations
and material needed for our work and state an existence result, which can be
proved by combining the techniques used in \cite{Caval01} and \cite%
{Boukhatem}. In section 3, we prove that, for suitable initial data, the
solution exists globally in time. In this section, our proof technique
follows the arguments of \cite{Shun15}, with some modifications being needed
for our problem. Section 4 contains the statement and the proof of our main
result.

\section{ Preliminaries and some notations}

In this section, we present some notations and some material needed in the
proof of our result.

Let $V=\left\{ u\in H^{1}\left( \Omega \right) :u=0\text{ on }\Gamma
_{0}\right\} $ be the subspace of the classical Sobolev space $H^{1}\left(
\Omega \right) $ of real valued functions of order one. The Poincare's
inequality holds on $V$, i.e. there exists a positive constant $C_{\ast }$
(depends only on $\Omega $) such that: 
\begin{equation}
\left\Vert u\right\Vert _{k}\leq C_{\ast }\left\Vert \nabla u\right\Vert
_{2},\text{ \ \ for every }u\in V,~\text{and }1\leq k\leq \overline{k}
\label{Poinca-inega}
\end{equation}%
where $\left\Vert u\right\Vert _{k}^{k}=\int_{\Omega }\left\vert
u\right\vert ^{k}dx$, $\left\Vert \nabla u\right\Vert _{2}^{2}=\int_{\Omega
}\left\vert \nabla u\right\vert ^{2}dx$ and $\overline{k}=\left\{ 
\begin{array}{l}
\frac{2n}{n-2}~\text{, if }n\geq 3 \\ 
+\infty ~\text{, if }n=1\text{ or }2%
\end{array}%
\right. $.\newline

According to the trace theory, there exists a positive constant $\overline{C}%
_{\ast }$ (depends only on $\Gamma _{1}$) such that: 
\begin{equation}
\left\Vert u\right\Vert _{2,\Gamma _{1}}\leq \overline{C}_{\ast }\left\Vert
\nabla u\right\Vert _{2},\text{ for every }u\in V,  \label{2Poinca-inega}
\end{equation}%
where $\left\Vert u\right\Vert _{2,\Gamma _{1}}^{2}=\int_{\Gamma
_{1}}\left\vert u\right\vert ^{2}dx$.\bigskip

To prove our result, we need the following assumptions.\medskip \newline
\textbf{(H1)} For the functions $p$ and $q$, we assume that $p,q\in C(\Gamma
_{1})$ and $p(x)>0$, $q(x)>0$, for all $x\in \Gamma _{1}$. This assumption
assures us that there exist positive constants $p_{i},q_{i}\ (i\in \left\{
0,1\right\} )$ such that: 
\begin{equation}
0<p_{0}\leq p(x)\leq p_{1};\ 0<q_{0}\leq q(x)\leq q_{1},\ \text{for all }%
x\in \Gamma _{1}.  \label{cond-on-p-q}
\end{equation}%
\textbf{(H2)} For the relaxation function $g:\mathbb{R}_{+}\rightarrow 
\mathbb{R}_{+}$, we assume that :%
\begin{equation}
g\left( 0\right) >0\ \ \ \text{ and }\ \ a-\int_{0}^{\infty }g\left( \
s\right) ds=l>0.  \label{Hypothese on g}
\end{equation}

and there exists a locally absolutely continuous function\ \ $\xi :\left[
0,+\infty \right) \rightarrow \left( 0,+\infty \right) $, and constants $%
\theta \geq 0,r$ such that

\begin{equation}
\int_{0}^{\infty }\xi (s)ds=+\infty ,\ \ \ g^{\prime }(t)\leq -\xi (t)g(t),%
\text{ for all }t\geq 0,.  \label{Old hypothesis on g' and ksi}
\end{equation}

\begin{equation}
\frac{\xi ^{\prime }}{\xi ^{\theta }}\in L^{1}\left( 0,+\infty \right) ,\ \
\ \int_{t}^{t+s}\frac{\xi ^{\prime }(\tau )}{\xi (\tau )}d\tau \leq r,\text{
for all }t,s\geq 0,.  \label{New hypothesis on ksi}
\end{equation}

The last assumption assures us that $\xi \left( t+s\right) \leq e^{r}\xi
\left( t\right) ,$ for all $t,s\geq 0,$ and so $\xi \in L^{\infty }\left(
0,+\infty \right) .$

\begin{remark}
(i) Notice that, if $\xi $ is a nonincreasing function, then hypothesis (\ref%
{New hypothesis on ksi}) is trivially valid with $\theta =$ $r=0$.

(ii) Notice also that, proofs in previous papers depend strongly on the
nonincreasence of $\xi .$ We prove our result without this
restriction.\bigskip
\end{remark}

We now state a local existence theorem for problem (\ref{Prob}), whose proof
follows the arguments in \cite{Boukhatem}.

\begin{theorem}[Local existence]
\label{th local exist}Let $2<k\leq \frac{2n-2}{n-2}$ , $(u_{0},u_{1})\in $ $%
V\times L^{2}(\Omega )$ $\ $and $y_{0}\in L^{2}(\Gamma _{1})$. Suppose that
hypotheses (H1)--(H2) hold, then there exists a unique pair of functions $%
(u,y)$, which is a solution of the problem (\ref{Prob}) such that 
\begin{equation*}
\begin{tabular}{ll}
$u\in C(0,T;V),$ & $u_{t}\in C(0,T;L^{2}(\Omega )),$ \\ 
$y\in L^{\infty }(0,T;L^{2}(\Gamma _{1})),\ \ \ $ & $y_{t}\in
L^{2}(0,T;L^{2}(\Gamma _{1}));$%
\end{tabular}%
\end{equation*}%
for some $T>0$.
\end{theorem}

\section{Global existence}

In this section, we shall state and prove the global existence. For this
purpose, we define the energy of problem (\ref{Prob}) by 
\begin{equation}
E\left( t\right) =\frac{1}{2}\left\Vert u_{t}\right\Vert _{2}^{2}+J\left(
u,y\right) \left( t\right) ,  \label{E-formula}
\end{equation}%
where 
\begin{multline}
J\left( u,y\right) \left( t\right) =\frac{1}{2}\left( a-\left(
\int_{0}^{t}g\left( s\right) ds\right) \right) \left\Vert \nabla
u\right\Vert _{2}^{2}+\frac{b}{\gamma +1}\left\Vert \nabla u\right\Vert
_{2}^{2\left( \gamma +1\right) }  \label{J- formula} \\
+\int_{\Gamma _{1}}\frac{1}{2}\left( q\left( x\right) y^{2}\right) dx+\frac{1%
}{2}\left( g\diamond \nabla u\right) \left( t\right) -\frac{1}{k}\left\Vert
u\right\Vert _{k}^{k},
\end{multline}%
and 
\begin{equation}
\left( g\diamond \nabla u\right) \left( t\right) =\int_{0}^{t}g\left(
t-s\right) \left\Vert \nabla u\left( t\right) -\nabla u\left( s\right)
\right\Vert _{2}^{2}ds.  \label{def-losenge}
\end{equation}

\begin{lemma}
\label{Lemma of E' formula} Let $(u,y)$ be the solution of (\ref{Prob}).
Then, the energy functional defined by (\ref{E-formula}) is a nonincreasing
function and 
\begin{equation}
E^{\prime }\left( t\right) =-\frac{1}{2}g\left( t\right) \left\Vert \nabla
u\right\Vert _{2}^{2}+\frac{1}{2}\left( g^{\prime }\diamond \nabla u\right)
\left( t\right) -\int_{\Gamma _{1}}p\left( x\right) y_{t}^{2}dx\text{, for
all }t\geq 0.  \label{E'-formula}
\end{equation}
\end{lemma}

\begin{proof}
{} Multiplying the first equation in (\ref{Prob}) by $u_{t}$, integrating
over $\Omega $, using integration by parts and exploiting the third equation
in system (\ref{Prob}), we get 
\begin{multline}
\frac{d}{dt}\left( \frac{1}{2}\left\Vert u_{t}\right\Vert _{2}^{2}+\frac{1}{2%
}a\left\Vert \nabla u\right\Vert _{2}^{2}+\frac{b}{2\left( \gamma +1\right) }%
\left\Vert \nabla u\right\Vert _{2}^{2\left( \gamma +1\right) }-\frac{1}{2}%
\left( \int_{0}^{t}g\left( \ s\right) ds\right) \left\Vert \nabla
u\right\Vert _{2}^{2}\right.  \label{formula1-proof-E'} \\
\left. +\frac{1}{2}\left( g\diamond \nabla u\right) \left( t\right) -\frac{1%
}{k}\left\Vert u\right\Vert _{k}^{k}\right) =\int_{\Gamma _{1}}u_{t}y_{t}dx-%
\frac{1}{2}g\left( t\right) \left\Vert \nabla u\right\Vert _{2}^{2}dx+\frac{1%
}{2}\left( g^{\prime }\diamond \nabla u\right) \left( t\right)
\end{multline}%
From the fourth equation in (\ref{Prob}), we deduice that 
\begin{equation}
\int_{\Gamma _{1}}u_{t}y_{t}dx=-\int_{\Gamma _{1}}p\left( x\right)
y_{t}^{2}dx-\frac{d}{dt}\left( \int_{\Gamma _{1}}\frac{1}{2}\left( q\left(
x\right) y^{2}\right) dx\right) .  \label{formula2-proof-E'}
\end{equation}%
Plugging (\ref{formula2-proof-E'}) into (\ref{formula1-proof-E'}) and making
use of (\ref{E-formula}), then we obtain (\ref{E'-formula}), and hence $%
E^{\prime }\left( t\right) \leq 0,\ \ $for all $t\geq 0$.\bigskip
\end{proof}

As in \cite{Caval07} and \cite{Shun15}, we define a functional $F$, which
helps in establishing the global existence of solution.

Setting

\begin{equation}
F(x)=\frac{1}{2}x^{2}-\frac{B_{\Omega }^{k}}{k}x^{k}  \label{Def of F}
\end{equation}

where

\begin{equation}
B_{\Omega }=\underset{u\in V,\ u\neq 0}{\sup }\frac{\left\Vert u\right\Vert
_{k}}{\sqrt{l\left\Vert \nabla u\right\Vert _{2}^{2}+\frac{b}{\gamma +1}%
\left\Vert \nabla u\right\Vert _{2}^{2\left( \gamma +1\right) }}}
\label{Def of  B}
\end{equation}

We can verify, as in \cite{Caval07}, that the function $F$ is increasing in $%
\left( 0,\lambda _{1}\right) $ and decreasing in $\left( \lambda
_{1},+\infty \right) ,$ where $\lambda _{1}$ is the strictly positive zero
of the derivative function $F^{\prime },$ that is

\begin{equation}
\lambda _{1}=B_{\Omega }^{\frac{-k}{k-2}}.  \label{value of lamda1}
\end{equation}%
$F$ has a maximum at $\lambda _{1}$ with the maximum value

\begin{equation}
d_{1}=F(\lambda _{1})=\left( \frac{k-2}{2k}\right) B_{\Omega }^{\frac{-2k}{%
k-2}}  \label{F maximum value}
\end{equation}%
From (\ref{E-formula}), (\ref{J- formula}), (\ref{Def of B}), (\ref%
{Hypothese on g}) and (\ref{Def of F}), we have

\begin{multline}
E\left( t\right) \geq J\left( u,y\right) \left( t\right) =\frac{1}{2}\left(
a-\left( \int_{0}^{t}g\left( s\right) ds\right) \right) \left\Vert \nabla
u\right\Vert _{2}^{2}+\frac{b}{2\left( \gamma +1\right) }\left\Vert \nabla
u\right\Vert _{2}^{2\left( \gamma +1\right) }  \label{E(t)abovF(gamma)} \\
+\int_{\Gamma _{1}}\frac{1}{2}\left( q\left( x\right) y^{2}\right) dx+\frac{1%
}{2}\left( g\diamond \nabla u\right) \left( t\right) -\frac{1}{k}\left\Vert
u\right\Vert _{k}^{k} \\
\geq \frac{1}{2}\gamma ^{2}\left( t\right) -\frac{1}{k}B_{\Omega }^{k}\gamma
^{k}\left( t\right) =F\left( \gamma \left( t\right) \right) \\
\end{multline}

where 
\begin{equation*}
\gamma \left( t\right) =\sqrt{l\left\Vert \nabla u\right\Vert _{2}^{2}+\frac{%
b}{\left( \gamma +1\right) }\left\Vert \nabla u\right\Vert _{2}^{2\left(
\gamma +1\right) }+\int_{\Gamma _{1}}q\left( x\right) y^{2}dx+\left(
g\diamond \nabla u\right) \left( t\right) }.
\end{equation*}

If $\gamma (t)<\lambda _{1}$, then we get 
\begin{multline}
E\left( t\right) \geq F\left( \gamma \left( t\right) \right) =\gamma
^{2}\left( t\right) \left( \frac{1}{2}-\frac{1}{k}B_{\Omega }^{k}\gamma
^{k-2}\left( t\right) \right)  \label{Minoration of E with gamma} \\
>\gamma ^{2}\left( t\right) \left( \frac{1}{2}-\frac{1}{k}B_{\Omega
}^{k}\lambda _{1}^{k-2}\right) \\
>\frac{k-2}{2k}\gamma ^{2}\left( t\right) . \\
\end{multline}

\begin{lemma}
\label{Lemma of stable set} Let $2<k\leq \frac{2n-2}{n-2}$, $%
(u_{0},u_{1})\in $ $V\times L^{2}(\Omega )$, $y_{0}\in L^{2}(\Gamma _{1})$
and hypotheses (H1)--(H2) hold. Assume further that $E\left( 0\right) <d_{1}$
and\newline
$\gamma \left( 0\right) =\sqrt{l\left\Vert \nabla u_{0}\right\Vert _{2}^{2}+%
\frac{b}{\left( \gamma +1\right) }\left\Vert \nabla u_{0}\right\Vert
_{2}^{2\left( \gamma +1\right) }+\int_{\Gamma _{1}}q\left( x\right)
y_{0}^{2}dx}<\lambda _{1}.$ Then 
\begin{tabular}{c}
$E\left( t\right) <d_{1}$%
\end{tabular}%
$~$\newline
and~$\gamma \left( t\right) =\sqrt{l\left\Vert \nabla u\right\Vert _{2}^{2}+%
\frac{b}{\left( \gamma +1\right) }\left\Vert \nabla u\right\Vert
_{2}^{2\left( \gamma +1\right) }+\int_{\Gamma _{1}}q\left( x\right)
y^{2}dx+\left( g\diamond \nabla u\right) \left( t\right) }<\lambda _{1},$
for all $t\in \lbrack 0,T)$.
\end{lemma}

\begin{proof}
Using the fact that $E(t)$ is a non-increasing function and (\ref%
{E(t)abovF(gamma)}) , we get

\begin{equation*}
E(t)\leq E(0)<d_{1}~\text{and }F\left( \gamma \left( t\right) \right)
<d_{1}~\ \text{for all\ }t\in \lbrack 0,T).
\end{equation*}%
From (\ref{Minoration of E with gamma}) and the fact that $E(0)<d_{1}$, it
follows that there exist $\lambda _{0},\lambda _{2}$ such that $0<\lambda
_{0}<\lambda _{1}<\lambda _{2}$ and $F(\lambda _{0})=E(0)=F(\lambda _{2}).$
As $F\left( \gamma \left( 0\right) \right) \leq E(0)=F(\lambda _{0})$ and $%
\lambda _{0},\gamma \left( 0\right) \in $ $\left( 0,\lambda _{1}\right) $
where the function $F$ is increasing, then $\gamma \left( 0\right) \leq
\lambda _{0}.$\medskip

We argue by contradiction to prove that $\gamma \left( t\right) \leq \lambda
_{0}$,$\ $for all\ $t\in \lbrack 0,T).$\newline
Suppose that there exists $t^{\ast }\in (0,T)$ such that $\gamma (t^{\ast
})>\lambda _{0}.$ We have two cases.\medskip \newline
\underline{Case 1}: $\lambda _{0}<\gamma (t^{\ast })<\lambda _{1},$ then, by
virtue of the increasince of $F$ and the non-increasince of $E$, we get

$F\left( \gamma (t^{\ast })\right) >F\left( \lambda _{0}\right) =E(0)\geq
E(t^{\ast }),$ which contradicts (\ref{E(t)abovF(gamma)}).\medskip \newline
\underline{Case 2}: $\lambda _{1}\leq \gamma (t^{\ast }),$ then, by virtue
of the continuity of $\gamma $ on $(0,t^{\ast }),$ there exists $t_{0}\in
\left( 0,t^{\ast }\right) $ such that $\lambda _{0}<\gamma (t_{0})<\lambda
_{1},$ which yields a contradiction as in case 1.

Thus, $\gamma \left( t\right) \leq \lambda _{0}$, and so $\gamma \left(
t\right) \leq \lambda _{1}\ $for all\ $t\in \lbrack 0,T).$
\end{proof}

\begin{theorem}
\label{th global exit}Let $\left( u\left( t\right) ,y\left( t\right) \right) 
$ be the solution of (\ref{Prob}). If $(u_{0},u_{1})\in $ $V\times
L^{2}(\Omega ),\ y_{0}\in L^{2}(\Gamma _{1})$ such that, $\gamma \left(
0\right) =\sqrt{l\left\Vert \nabla u_{0}\right\Vert _{2}^{2}+\frac{b}{\left(
\gamma +1\right) }\left\Vert \nabla u_{0}\right\Vert _{2}^{2\left( \gamma
+1\right) }+\int_{\Gamma _{1}}q\left( x\right) y_{0}^{2}dx}<\lambda _{1}$
and $E\left( 0\right) <d_{1}$, then the solution $\left( u\left( t\right)
,y\left( t\right) \right) $ is global in time.
\end{theorem}

\begin{proof}
Using Lemma \ref{Lemma of stable set}, (\ref{Minoration of E with gamma}), (%
\ref{E(t)abovF(gamma)}) and (\ref{E-formula}), we get 
\begin{multline}
\frac{1}{2}\left\Vert u_{t}\right\Vert _{2}^{2}+\frac{k-2}{2k}\gamma
^{2}\left( t\right) \leq \frac{1}{2}\left\Vert u_{t}\right\Vert
_{2}^{2}+F\left( \gamma \left( t\right) \right) \\
\leq \frac{1}{2}\left\Vert u_{t}\right\Vert _{2}^{2}+J\left( u,y\right)
\left( t\right) =E(t)\leq E(0)<d_{1}, \\
\end{multline}

which ensures the boundedness of $u$ in $V,$ $u_{t}$ in $L^{2}(\Omega )$ and 
$y$ in $L^{2}(\Gamma _{1}),$ and hence the solution $\left( u\left( t\right)
,y\left( t\right) \right) $ is bounded and global in time.
\end{proof}

\section{Decay result}

In order to study the decay estimate of global solution for the problem (\ref%
{Prob}), we need the following lemma.

\begin{lemma}
\label{lemme Martinez}(\cite{Martinez}) Let $E:%
\mathbb{R}
_{+}\rightarrow 
\mathbb{R}
_{+}$ be a non-increasing function and \newline
$\phi :%
\mathbb{R}
_{+}\rightarrow 
\mathbb{R}
_{+}$ a strictly increasing $C^{1}$ function such that 
\begin{equation*}
\phi (0)=0\quad \text{and}\quad \phi (t)\rightarrow +\infty \quad \text{as}%
\quad t\rightarrow +\infty .
\end{equation*}%
Assume that there exist $\sigma \geq 0$ and $\omega >0$ such that 
\begin{equation*}
\int_{S}^{+\infty }E^{1+\sigma }(t)\phi ^{\prime }(t)\,dt\leq \frac{1}{%
\omega }E^{\sigma }(0)E(S).\quad \text{for all }0\leq S<+\infty ,
\end{equation*}%
then 
\begin{equation*}
E(t)\leq E(0)\left( \frac{1+\sigma }{1+\omega \sigma \phi (t)}\right) ^{%
\frac{1}{\sigma }}\ \ \ \text{for all }t\geq 0,\quad \text{if}\quad \sigma
>0,
\end{equation*}%
\begin{equation*}
E(t)\leq E(0)e^{1-\omega \phi (t)}\text{ \ for all }t\geq 0,\quad \text{if}%
\quad \sigma =0.
\end{equation*}
\end{lemma}

Our main result is the following

\begin{theorem}
Under assumptions of Theorem \ref{th local exist} and the assumption that $%
E\left( 0\right) <d_{1}$ and $\gamma \left( 0\right) <\lambda _{1}$, there
exists a positive constant $\omega $ depending on initial energy $E(0)$,
such that the solution energy of (\ref{Prob}) satisfies,

\begin{equation*}
E(t)\leq E(0)e^{1-\omega \int_{0}^{t}\xi (s)ds},\ \ \text{for all }t\geq 0
\end{equation*}
\end{theorem}

\begin{remark}
If $E(t_{0})=0$, for some $t_{0}\geq 0$; then from Lemma \ref{Lemma of E'
formula}, we have $E(t)=0,$ for all $t\geq t_{0}$, and then we have nothing
to prove in this case. So we assume that $E(t)>0,$ for all $t\geq 0$ without
loss of generality.\bigskip
\end{remark}

\begin{proof}
Since $g$ is positive and $g(0)>0$ then for any $t_{0}>0$ we have

\begin{equation}
\int_{0}^{t}g\left( s\right) ds\geq \int_{0}^{t_{0}}g\left( s\right)
ds=g_{0},\text{ for all }t\geq t_{0}  \label{def g0}
\end{equation}

Multiplying the first equation in (\ref{Prob}) by $\xi (t)u(t)$ and
integrating by parts over $\Omega \times (S,T),$ with $S\geq t_{0}$, we get

\begin{multline}
\int_{S}^{T}\int_{\Omega }\left( \xi (t)u.u_{t}\right) ^{\prime
}dxdt-\int_{S}^{T}\int_{\Omega }\xi ^{\prime
}(t)u.u_{t}dxdt-\int_{S}^{T}\int_{\Omega }\xi (t)u_{t}^{2}dxdt
\label{first egalite} \\
+\int_{S}^{T}\xi (t)\left( a-\int_{0}^{t}g\left( s\right) ds\right)
\left\Vert \nabla u\right\Vert _{2}^{2}dt+\int_{S}^{T}b\xi (t)\left\Vert
\nabla u\right\Vert _{2}^{2\left( \gamma +1\right) }dt \\
-\int_{S}^{T}\xi (t)\int_{0}^{t}g\left( t-s\right) \int_{\Omega }\left[
\nabla u\left( s\right) -\nabla u\left( t\right) \right] \nabla u\left(
t\right) dxdsdt \\
-\int_{S}^{T}\xi (t)\int_{\Gamma _{1}}y_{t}udxdt=\int_{S}^{T}\xi
(t)\left\Vert u\left( t\right) \right\Vert _{k}^{k}dt
\end{multline}

since $\gamma \geq 0,$ we deduce that

\begin{multline}
\int_{S}^{T}\xi (t)\left( a-\int_{0}^{t}g\left( s\right) ds\right)
\left\Vert \nabla u\right\Vert _{2}^{2}dt+\int_{S}^{T}\frac{b}{\left( \gamma
+1\right) }\xi (t)\left\Vert \nabla u\right\Vert _{2}^{2\left( \gamma
+1\right) }dt-\int_{S}^{T}\xi (t)\left\Vert u\left( t\right) \right\Vert
_{k}^{k}dt  \label{first estim} \\
\leq -\int_{S}^{T}\int_{\Omega }\left( \xi (t)u.u_{t}\right) ^{\prime
}dxdt+\int_{S}^{T}\xi ^{\prime }(t)\int_{\Omega
}u.u_{t}dxdt+\int_{S}^{T}\int_{\Omega }\xi (t)u_{t}^{2}dxdt \\
+\int_{S}^{T}\xi (t)\int_{0}^{t}g\left( t-s\right) \int_{\Omega }\left[
\nabla u\left( s\right) -\nabla u\left( t\right) \right] \nabla u\left(
t\right) dxdsdt+\int_{S}^{T}\xi (t)\int_{\Gamma _{1}}y_{t}udxdt
\end{multline}

Recalling the definition of $B_{\Omega }$ above and using (\ref{Minoration
of E with gamma}), we can estimate

\begin{multline}
\frac{\left\Vert u\right\Vert _{k}^{k}}{\left( \left( a-\int_{0}^{t}g\left(
s\right) ds\right) \left\Vert \nabla u\right\Vert _{2}^{2}+\frac{b}{\left(
\gamma +1\right) }\left\Vert \nabla u\right\Vert _{2}^{2\left( \gamma
+1\right) }\right) } \\
\leq \left( \frac{\left\Vert u\left( t\right) \right\Vert _{k}}{\sqrt{%
l\left\Vert \nabla u\right\Vert _{2}^{2}+\frac{b}{\left( \gamma +1\right) }%
\left\Vert \nabla u\right\Vert _{2}^{2\left( \gamma +1\right) }}}\right)
^{k}\left( \frac{2k}{k-2}E\left( t\right) \right) ^{\frac{k}{2}-1} \\
\leq B_{\Omega }^{k}\left( \frac{2k}{k-2}E\left( 0\right) \right) ^{\frac{k}{%
2}-1} \\
<B_{\Omega }^{k}\left( \frac{2k}{k-2}\left( \frac{k-2}{2k}\right) B_{\Omega
}^{\frac{-2k}{k-2}}\right) ^{\frac{k}{2}-1}=1 \\
\end{multline}

In view of this inequality there exists a constant $c>0$ independent of $t$,
such that

\begin{multline}
c\left( \int_{S}^{T}\frac{1}{2}\xi (t)\left( a-\int_{0}^{t}g\left( s\right)
ds\right) \left\Vert \nabla u\right\Vert _{2}^{2}dt+\int_{S}^{T}\xi (t)\frac{%
b}{2\left( \gamma +1\right) }\left\Vert \nabla u\right\Vert _{2}^{2\left(
\gamma +1\right) }dt\right.  \label{majoration Cintksi E(t)} \\
\left. -\frac{1}{k}\int_{S}^{T}\xi (t)\left\Vert u\left( t\right)
\right\Vert _{k}^{k}dt\right) \leq \int_{S}^{T}\xi (t)\left(
a-\int_{0}^{t}g\left( s\right) ds\right) \left\Vert \nabla u\right\Vert
_{2}^{2}dt \\
+\int_{S}^{T}\frac{b}{\left( \gamma +1\right) }\xi (t)\left\Vert \nabla
u\right\Vert _{2}^{2\left( \gamma +1\right) }dt-\int_{S}^{T}\xi
(t)\left\Vert u\left( t\right) \right\Vert _{k}^{k}dt.
\end{multline}

Adding some terms to both sides of the obove inequality and using (\ref%
{first estim}), we deduce

\begin{multline}
c\int_{S}^{T}\xi (t)E\left( t\right) dt\leq -\int_{S}^{T}\left( \xi
(t)\int_{\Omega }u.u_{t}dx\right) ^{\prime }dxdt+\int_{S}^{T}\xi ^{\prime
}(t)\int_{\Omega }u.u_{t}dxdt  \label{1ere estim de CInt ksi E} \\
+\int_{S}^{T}\xi (t)\int_{0}^{t}g\left( t-s\right) \int_{\Omega }\left[
\nabla u\left( s\right) -\nabla u\left( t\right) \right] \nabla u\left(
t\right) dxdsdt \\
+\left( 1+\frac{c}{2}\right) \int_{S}^{T}\xi (t)\left\Vert u_{t}\right\Vert
_{2}^{2}dt+\int_{S}^{T}\xi (t)\int_{\Gamma _{1}}y_{t}udxdt \\
+c\frac{1}{2}\int_{S}^{T}\xi (t)\int_{\Gamma _{1}}q\left( x\right)
y^{2}dxdt+c\frac{1}{2}\int_{S}^{T}\xi (t)\left( g\diamond \nabla u\right)
\left( t\right) dt
\end{multline}

Using the Cauchy-Schwarz's inequality, the Poincar\'{e}'s inequalities (\ref%
{Poinca-inega}), (\ref{Hypothese on g}) 
and the definition of energy (\ref{E-formula}), we obtain estimates as
follows

\begin{multline}
\left\vert \int_{\Omega }uu_{t}dx\right\vert \leq \sqrt{\frac{k}{l\left(
k-2\right) }}C_{\ast }\left\vert \int_{\Omega }u_{t}\sqrt{\frac{l\left(
k-2\right) }{k}}\frac{1}{C_{\ast }}udx\right\vert  \label{estim Hayh1} \\
\leq \frac{1}{2}\sqrt{\frac{k}{l\left( k-2\right) }}C_{\ast }\left(
\int_{\Omega }u_{t}^{2}dx+\frac{k-2}{k}\frac{l}{C_{\ast }^{2}}\int_{\Omega
}u^{2}dx\right) \\
\leq \frac{1}{2}\sqrt{\frac{k-2}{lk}}C_{\ast }\left( \left\Vert
u_{t}\right\Vert _{2}^{2}+\frac{k-2}{k}l\left\Vert \nabla u\right\Vert
_{2}^{2}\right) \\
\leq \sqrt{\frac{k-2}{lk}}C_{\ast }E\left( t\right) . \\
\end{multline}

Now, using Cauchy's inequality, (\ref{E-formula}), (\ref{Hypothese on g})
and a technique used in \cite{Messaoudi08}, we get

\begin{multline}
\left\vert \int_{0}^{t}g\left( t-s\right) \int_{\Omega }\left[ \nabla
u\left( s\right) -\nabla u\left( t\right) \right] \nabla u\left( t\right)
dxds\right\vert \leq \epsilon \left\Vert \nabla u\left( t\right) \right\Vert
_{2}^{2}  \label{estim Hayh2} \\
+\frac{1}{4\epsilon }\int_{\Omega }\left( \int_{0}^{t}g\left( t-s\right)
\left( \nabla u\left( s\right) -\nabla u\left( t\right) \right) ds\right)
^{2}dx \\
\leq \epsilon \left\Vert \nabla u\left( t\right) \right\Vert _{2}^{2}+\frac{1%
}{4\epsilon }\int_{\Omega }\left( \int_{0}^{t}\!g(t-s)ds\right)
\int_{0}^{t}g(t-s)\left( \nabla u\left( s\right) -\nabla u\left( t\right)
\right) ^{2}dsdx \\
\leq \epsilon \frac{2k}{\left( k-2\right) l}E\left( t\right) +\frac{1}{%
4\epsilon }\left( a-l\right) \left( g\diamond u\right) \left( t\right) , \\
\end{multline}

for any $\epsilon >0.$

The following estimate is obtained using the trace theory, hypotesis (\ref%
{cond-on-p-q}), Poincar\'{e}'s inequalitie (\ref{2Poinca-inega}), H\"{o}%
lder's and Cauchy's inequalities

\begin{multline}
\left\vert \int_{\Gamma _{1}}u\left( t\right) y_{t}\left( t\right)
dx\right\vert \leq \frac{1}{\epsilon }\int_{\Gamma _{1}}y_{t}^{2}dx+\epsilon
\int_{\Gamma _{1}}u^{2}dx  \label{estim Hayh} \\
\leq \frac{1}{\epsilon }\frac{1}{p_{0}}\int_{\Gamma _{1}}p\left( x\right)
y_{t}^{2}dx+\epsilon \overline{C}_{\ast }^{2}\left\Vert \nabla u\right\Vert
_{2}^{2} \\
\leq \frac{1}{\epsilon }\frac{1}{p_{0}}\left( -E^{\prime }\left( t\right)
\right) +\epsilon \overline{C}_{\ast }^{2}\frac{2k}{\left( k-2\right) l}%
E\left( t\right) , \\
\end{multline}%
for any $\epsilon >0,$

To estimate the sixth term of (\ref{1ere estim de CInt ksi E}), we multiply
the fourth equation in (\ref{Prob}) by $\xi (t)y$, integrate by parts over $%
\left( S,T\right) \times \Gamma _{1}$ and use Cauchy's inequalitie:

\begin{multline}
\int_{S}^{T}\int_{\Gamma _{1}}\xi (t)q\left( x\right) y^{2}dxdt\leq
-\int_{S}^{T}\int_{\Gamma _{1}}\left( \xi (t)yu\right) ^{\prime
}dxdt+\int_{S}^{T}\int_{\Gamma _{1}}\xi ^{\prime }(t)yudxdt+\int_{S}^{T}\xi
(t)\int_{\Gamma _{1}}y_{t}udxdt \\
+\frac{1}{2}\int_{S}^{T}\int_{\Gamma _{1}}\xi (t)q\left( x\right) y^{2}dxdt+%
\frac{1}{2}\left\Vert \xi \right\Vert _{\infty }\frac{p_{1}}{q_{0}}%
\int_{S}^{T}\int_{\Gamma _{1}}p\left( x\right) y_{t}^{2}dxdt
\end{multline}

Using hypotesis (\ref{cond-on-p-q}), and (\ref{E'-formula}), we get

\begin{multline}
\frac{1}{2}\int_{S}^{T}\int_{\Gamma _{1}}\xi (t)q\left( x\right)
y^{2}dxdt\leq \left\vert \xi (S)\int_{\Gamma _{1}}y\left( S\right) u\left(
S\right) \right\vert +\left\vert \xi (T)\int_{\Gamma _{1}}y\left( T\right)
u\left( T\right) \right\vert +\int_{S}^{T}\left\vert \xi ^{\prime
}(t)\right\vert \left\vert \int_{\Gamma _{1}}yudx\right\vert
\label{estim avant avant HOOH1} \\
+\int_{S}^{T}\xi (t)\int_{\Gamma _{1}}y_{t}udxdt+\frac{1}{2}\left\Vert \xi
\right\Vert _{\infty }\frac{p_{1}}{q_{0}}\int_{S}^{T}\left( -E^{\prime
}\left( t\right) \right) dt
\end{multline}

Using the trace theory, (\ref{2Poinca-inega}), (\ref{E-formula}), H\"{o}%
lder's and Cauchy's inequalities, we get

\begin{multline}
\left\vert \int_{\Gamma _{1}}uydx\right\vert \leq \int_{\Gamma _{1}}\frac{1}{%
2q\left( x\right) }u^{2}dx+\int_{\Gamma _{1}}\frac{q\left( x\right) }{2}%
y^{2}dx  \label{estim int uy} \\
\leq \frac{\overline{C}_{\ast }^{2}}{2q_{0}}\left\Vert \nabla u\right\Vert
_{2}^{2}+\int_{\Gamma _{1}}\frac{q\left( x\right) }{2}y^{2}dx \\
\leq \frac{k}{\left( k-2\right) }\left( \frac{\overline{C}_{\ast }^{2}}{%
q_{0}l}+1\right) E\left( t\right) , \\
\end{multline}

Combining now (\ref{estim Hayh}), (\ref{estim avant avant HOOH1}), (\ref%
{estim int uy}) and the fact that $E$ is decreasing and $\xi $ is bounded,
we get

\begin{equation}
\int_{S}^{T}\int_{\Gamma _{1}}\xi (t)q\left( x\right) y^{2}dxdt\leq
c_{0}\left( \epsilon \right) E(S)+\epsilon c_{1}\int_{S}^{T}\xi (t)E\left(
t\right) dt  \label{estim finale HOOH1}
\end{equation}%
where $c_{0}\left( \epsilon \right) =4\left( \left\Vert \xi \right\Vert
_{\infty }\frac{k}{\left( k-2\right) }\left( \frac{\overline{C}_{\ast }^{2}}{%
q_{0}l}+1\right) +\frac{k}{\left( k-2\right) }\left( \frac{\overline{C}%
_{\ast }^{2}}{q_{0}l}+1\right) \left\Vert \xi \right\Vert _{\infty }^{\theta
}\left\Vert \frac{\xi ^{\prime }}{\xi ^{\theta }}\right\Vert _{1}+\left\Vert
\xi \right\Vert _{\infty }\frac{p_{1}}{q_{0}}+\left\Vert \xi \right\Vert
_{\infty }\frac{1}{\epsilon }\frac{1}{p_{0}}\right) \ $\newline
and $c_{1}=\overline{C}_{\ast }^{2}\frac{4k}{\left( k-2\right) l}.$\medskip

To estimate the third term of (\ref{1ere estim de CInt ksi E}), we multiply
the first equation in (\ref{Prob}) by $\xi
(t)\int_{0}^{t}g(t-s)(u(t)-u(s))ds $, integrate by parts over $\Omega \times
(S,T)$ and utilize (\ref{def g0}):

\begin{multline}
g_{0}\int_{S}^{T}\xi (t)\left\Vert u_{t}\right\Vert _{2}^{2}\leq
\int_{S}^{T}\left( \xi (t)\int_{\Omega
}\int_{0}^{t}g(t-s)(u(t)-u(s))ds.u_{t}\right) ^{\prime }dxdt
\label{estim Norme ut^2 first} \\
-\int_{S}^{T}\xi ^{\prime }(t)\int_{\Omega }\left(
\int_{0}^{t}g(t-s)(u(t)-u(s))ds\right) .u_{t}dxdt \\
-\int_{S}^{T}\xi (t)\int_{\Omega }u_{t}\int_{0}^{t}g^{\prime
}(t-s)(u(t)-u(s))dsdxdt \\
+\int_{S}^{T}\left( \xi (t)\left( \left( a+b\left\Vert \nabla u\right\Vert
_{2}^{2\gamma }\right) -\int_{0}^{t}g\left( s\right) ds\right)
\int_{0}^{t}g(t-s)\int_{\Omega }\nabla u(\nabla u(t)-\nabla u(s))dxds\right)
dt \\
+\int_{S}^{T}\xi (t)\int_{\Omega }\left( \int_{0}^{t}g\left( t-s\right)
\left( \nabla u(t)-\nabla u\left( s\right) \right) ds\right) ^{2}dxdt \\
-\int_{S}^{T}\left( \xi (t)\int_{\Gamma _{1}}y_{t}\left(
\int_{0}^{t}g(t-\tau )(u(t)-u(\tau ))d\tau \right) dx\right) dt \\
-\int_{S}^{T}\xi (t)\int_{0}^{t}g(t-s)\int_{\Omega }\left\vert u\left(
t\right) \right\vert ^{k-2}u\left( t\right) \left( u(t)-u(s)\right) dsdxdt \\
\end{multline}

As above, we can obtain the following estimates

\begin{multline}
\left\vert \int_{\Omega }\int_{0}^{t}g(t-s)(u(t)-u(s))u_{t}\left( t\right)
dsdx\right\vert \leq \frac{1}{2}\int_{\Omega }u_{t}^{2}\left( t\right) dx
\label{estim TOZ-2} \\
+\frac{1}{2}\int_{\Omega }\left( \int_{0}^{t}g(t-s)\left( u(t)-u(s)\right)
ds\right) ^{2}dx \\
\leq \frac{1}{2}\left\Vert u_{t}\right\Vert _{2}^{2}+\frac{1}{2}\left(
a-l\right) \!\int_{0}^{t}\!g(t-s)\!\int_{\Omega }\left( u(t)-u(s)\right)
^{2}dxds \\
\leq \frac{1}{2}\left\Vert u_{t}\right\Vert _{2}^{2}+\frac{1}{2}\left(
a-l\right) C_{\ast }^{2}\left( g\diamond u\right) \left( t\right) \\
\leq \left( 1+\left( a-l\right) C_{\ast }^{2}\frac{k}{k-2}\right) E(t), \\
\end{multline}

\begin{multline}
\left\vert \int_{\Omega }u_{t}\left( t\right) \int_{0}^{t}\left( g^{\prime
}(t-s)(u(t)-u(s))\right) dsdx\right\vert \leq \epsilon \int_{\Omega }\left(
u_{t}\left( t\right) \right) ^{2}dx  \label{estim TOZ-1} \\
+\frac{1}{4\epsilon }\int_{\Omega }\left( \int_{0}^{t}\left( g^{\prime
}(t-s)(u(t)-u(s))\right) ds\right) ^{2}dx \\
\leq \epsilon \left\Vert u_{t}\left( t\right) \right\Vert _{2}^{2}-\frac{1}{%
4\epsilon }g(0)\int_{0}^{t}g^{\prime }(t-s)\int_{\Omega }\left(
u(t)-u(s)\right) ^{2}dxds \\
\leq \epsilon \left\Vert u_{t}\left( t\right) \right\Vert _{2}^{2}-\frac{%
C_{\ast }^{2}}{4\epsilon }g(0)\left( g^{\prime }\diamond u\right) \left(
t\right) \\
\leq \epsilon 2E(t)+\frac{C_{\ast }^{2}}{2\epsilon }g(0)\left( -E^{\prime
}\left( t\right) \right) , \\
\end{multline}

and

\begin{multline}
\left\vert \int_{0}^{t}g(t-s)\int_{\Omega }\nabla u\left( t\right) (\nabla
u(t)-\nabla u(s))dxds\right\vert \leq \int_{\Omega }\epsilon \left( \nabla
u\left( t\right) \right) ^{2}dx  \label{TOZ0} \\
+\frac{1}{4\epsilon }\left( \int_{0}^{t}g(s)ds\right) \left( \int_{\Omega
}\int_{0}^{t}g(t-s)(\nabla u(t)-\nabla u(s))^{2}dsdx\right) \\
\leq \left( \epsilon \left\Vert \nabla u\right\Vert _{2}^{2}+\frac{\left(
a-l\right) }{4\epsilon }\left( g\diamond u\right) \left( t\right) \right) ,
\\
\end{multline}

Combining the estimate (\ref{TOZ0}), the fact that $\left\Vert \nabla
u\right\Vert _{2}^{2}\leq \frac{2k}{\left( k-2\right) l}E\left( t\right) $
and again $\int_{0}^{t}g(s)ds\leq a-l,$ we get

\begin{multline}
\left( \left( a+b\left\Vert \nabla u\right\Vert _{2}^{2\gamma }\right)
-\int_{0}^{t}g(s)ds\right) \int_{0}^{t}g(t-s)\int_{\Omega }\nabla u\left(
t\right) (\nabla u(t)-\nabla u(s))dxds  \label{TOZ 0+} \\
\leq \left( 2a-l+b\left( \frac{2k}{\left( k-2\right) l}E\left( 0\right)
\right) ^{\gamma }\right) \left( \epsilon \frac{2k}{\left( k-2\right) l}%
E\left( t\right) +\frac{\left( a-l\right) }{4\epsilon }\left( g\diamond
u\right) \left( t\right) \right) ,
\end{multline}

The trace theory, hypotesis (\ref{cond-on-p-q}), Poincar\'{e}'s inequalitie (%
\ref{2Poinca-inega}), H\"{o}lder's and Cauchy's inequalities, permit us to
get

\begin{multline}
\left\vert \int_{0}^{t}g(t-\tau )\int_{\Gamma }y_{t}\left( t\right)
(u(t)-u(\tau ))dxd\tau \right\vert \leq \frac{1}{2}\int_{\Gamma
_{1}}y_{t}^{2}\left( t\right) dx  \label{estim TOZ2} \\
+\frac{1}{2}\int_{\Gamma _{1}}\left( \int_{0}^{t}g(t-\tau )\left(
u(t)-u(\tau )\right) d\tau \right) ^{2}dx \\
\leq \frac{1}{2}\frac{1}{p_{0}}\int_{\Gamma _{1}}p\left( x\right)
y_{t}^{2}dx+\frac{1}{2}\left( a-l\right) \int_{0}^{t}g(t-\tau )\int_{\Gamma
_{1}}\left( \left( u(t)-u(\tau )\right) \right) ^{2}dxd\tau  \\
\leq \frac{1}{2}\frac{1}{p_{0}}\left( -E^{\prime }\left( t\right) \right) +%
\frac{1}{2}\left( a-l\right) \overline{C}_{\ast }^{2}\left( g\diamond
u\right) \left( t\right) . \\
\end{multline}%
Using the fact that $2<k\leq \frac{2n-2}{n-2}$ and Poincar\'{e}'s
inequalitie (\ref{Poinca-inega}), the last term on the right hand side of (%
\ref{estim Norme ut^2 first}), can be estimated as follows

\begin{multline}
\int_{0}^{t}g(t-s)\int_{\Omega }\left\vert u\left( t\right) \right\vert
^{k-2}u\left( t\right) \left( u(t)-u(s)\right) dxds\leq \epsilon
\int_{\Omega }\left\vert u\left( t\right) \right\vert ^{2\left( k-1\right)
}dx  \label{estim TOZ} \\
+\frac{1}{4\epsilon }\int_{\Omega }\left( \int_{0}^{t}g(t-s)\left(
u(t)-u(s)\right) ds\right) ^{2}dx \\
\leq \epsilon \left( C_{\ast }\left\Vert \nabla u\right\Vert _{2}\right)
^{2\left( k-1\right) }+\frac{1}{4\epsilon }\left( a-l\right) C_{\ast
}^{2}\left( g\diamond u\right) \left( t\right) \\
\leq \epsilon C_{\ast }^{2\left( k-1\right) }\left( \frac{2k}{\left(
k-2\right) l}E\left( t\right) \right) ^{k-1}+\frac{1}{4\epsilon }\left(
a-l\right) C_{\ast }^{2}\left( g\diamond u\right) \left( t\right) \\
\leq \epsilon \left( \frac{2kC_{\ast }^{2}}{\left( k-2\right) l}\right)
^{k-1}\left( E\left( 0\right) \right) ^{k-2}E\left( t\right) \\
+\frac{1}{4\epsilon }\left( a-l\right) C_{\ast }^{2}\left( g\diamond
u\right) \left( t\right) . \\
\end{multline}

Combining estimates (\ref{estim Norme ut^2 first})-(\ref{estim TOZ}), we get

\begin{equation}
\int_{S}^{T}\xi (t)\left\Vert u_{t}\right\Vert _{2}^{2}\leq c_{2}\left(
\epsilon \right) E(S)+\epsilon c_{3}\int_{S}^{T}\xi (t)E\left( t\right)
dt+c_{4}\left( \epsilon \right) \int_{S}^{T}\xi (t)\left( g\diamond u\right)
\left( t\right) dt  \label{estim finale de ut^2}
\end{equation}%
where $c_{2}\left( \epsilon \right) =g_{0}^{-1}\left( \left( 2\xi _{\infty
}+\left\Vert \xi \right\Vert _{\infty }^{\theta }\left\Vert \frac{\xi
^{\prime }}{\xi ^{\theta }}\right\Vert _{1}\right) \left( 1+\left(
a-l\right) C_{\ast }^{2}\frac{k}{k-2}\right) +\frac{\xi _{\infty }}{2}\left( 
\frac{C_{\ast }^{2}}{\epsilon }g(0)+\frac{1}{p_{0}}\right) \right) ,$\newline
$c_{3}=g_{0}^{-1}\left( 2+\frac{2k\left( 2a-l+b\left( \frac{2k}{\left(
k-2\right) l}E\left( 0\right) \right) ^{\gamma }\right) }{\left( k-2\right) l%
}+\left( \frac{2kC_{\ast }^{2}}{\left( k-2\right) l}\right) ^{k-1}\left(
E\left( 0\right) \right) ^{k-2}\right) $ and \newline
$c_{4}\left( \epsilon \right) =g_{0}^{-1}\left( \frac{\left( 2a-l+b\left( 
\frac{2k}{\left( k-2\right) l}E\left( 0\right) \right) ^{\gamma }\right)
\left( a-l\right) }{4\epsilon }+\left( a-l\right) \left( 1+\frac{1}{2}%
\overline{C}_{\ast }^{2}+\frac{1}{4\epsilon }C_{\ast }^{2}\right) \right) .$

From hypotesis (\ref{Old hypothesis on g' and ksi}) and (\ref{New hypothesis
on ksi}), it follows that

\begin{multline}
\int_{S}^{T}\xi (t)\left( g\diamond \nabla u\right) \left( t\right) dt\leq
e^{r}\int_{S}^{T}\int_{0}^{t}\xi (t-s)g\left( t-s\right) \left\Vert \nabla
u\left( t\right) -\nabla u\left( s\right) \right\Vert _{2}^{2}dsdt
\label{estim ksi losange} \\
\leq e^{r}\int_{S}^{T}-\left( g^{\prime }\diamond \nabla u\right) \left(
t\right) dt \\
\leq e^{r}\int_{S}^{T}\left( -2E^{\prime }\left( t\right) \right) dt \\
\leq 2e^{r}E\left( S\right) \\
\end{multline}

Consequently, using (\ref{1ere estim de CInt ksi E}), (\ref{estim Hayh1}), (%
\ref{estim Hayh2}), (\ref{estim finale de ut^2}), (\ref{estim Hayh}), (\ref%
{estim finale HOOH1}), (\ref{estim ksi losange}), (\ref{New hypothesis on
ksi}) and the fact that $E$ is decreasing , we get

\begin{equation}
c\int_{S}^{T}\xi (t)E\left( t\right) dt\leq c_{5}\left( \epsilon \right)
E\left( S\right) +\epsilon c_{6}\int_{S}^{T}\xi (t)E\left( t\right) dt
\end{equation}%
where $c_{5}\left( \epsilon \right) =\sqrt{\frac{k-2}{lk}}\left( 2\left\Vert
\xi \right\Vert _{\infty }C_{\ast }+C_{\ast }\left\Vert \xi \right\Vert
_{\infty }^{\theta }\left\Vert \frac{\xi ^{\prime }}{\xi ^{\theta }}%
\right\Vert _{1}\right) +c_{2}\left( \epsilon \right) \left( 1+\frac{c}{2}%
\right) +\left( \frac{1}{2\epsilon }\left( a-l\right) +2c_{4}\left( \epsilon
\right) \left( 1+c\right) \right) e^{r}$

$\ \ \ \ \ \ \ \ \ \ \ +\frac{1}{\epsilon }\frac{1}{p_{0}}\left\Vert \xi
\right\Vert _{\infty }+\frac{c}{2}c_{0}\left( \epsilon \right) $ and $c_{6}=%
\frac{2k}{\left( k-2\right) l}\left( 1+\overline{C}_{\ast }^{2}\right)
+c_{3}+\left( c_{1}+c_{3}\right) \frac{c}{2}$ \newline
and so, for $\epsilon $ small enough, we get

\begin{equation}
\int_{S}^{T}\xi (t)E\left( t\right) dt\leq \frac{c_{5}\left( \epsilon
\right) }{\left( c-\epsilon c_{6}\right) }E\left( S\right) ,\ \text{for all\ 
}S\geq t_{0},  \label{estim finale -1}
\end{equation}%
and if $0\leq S<t_{0},$ it suffices to observe that

\begin{multline}
\int_{S}^{T}\xi (t)E\left( t\right) dt=\int_{S}^{t_{0}}\xi (t)E\left(
t\right) dt+\int_{t_{0}}^{T}\xi (t)E\left( t\right) dt \\
\leq E\left( S\right) \int_{0}^{t_{0}}\xi (t)dt+\frac{c_{5}\left( \epsilon
\right) }{\left( c-\epsilon c_{6}\right) }E\left( t_{0}\right) \\
\end{multline}

therefore $\int_{S}^{T}\xi (t)E\left( t\right) dt\leq CE\left( S\right) ,\ $%
for all\ $S\geq 0\ $and some constant $C\geq 0$ independent of $S$ and $T$

Let $T\rightarrow +\infty ,$ applying Lemma \ref{lemme Martinez} (with $%
\sigma =0$ and $\phi \left( t\right) =\int_{0}^{t}\xi (\tau )d\tau $), we
conclude that $E(t)\leq E(0)e^{1-\omega \int_{0}^{t}\xi (\tau )d\tau },$\
for all $t\geq 0,$ for some $\omega =\omega \left( E\left( 0\right) ,\xi
;t_{0}\right) $
\end{proof}

\end{document}